\documentclass[a4paper]{article}

\usepackage{amsfonts}
\usepackage{graphicx}
\usepackage{amsthm}
\usepackage{amsmath}
\usepackage{amssymb}
\usepackage{enumerate}
\allowdisplaybreaks

\newtheorem{proposition}{Proposition}[section]
\newtheorem{lemma}[proposition]{Lemma}

\newtheorem{theorem}[proposition]{Theorem}
\newtheorem{definition}[proposition]{Definition}
\theoremstyle{definition}
\newtheorem{example}[proposition]{Example}

\numberwithin{equation}{section}

\begin{document}

\begin{center}
\LARGE
\textbf{The diameter of products of finite simple groups}
\bigskip\bigskip

\large
Daniele Dona\footnote{The author was partially supported by the European Research Council under Programme H2020-EU.1.1., ERC Grant ID: 648329 (codename GRANT).}
\bigskip

\normalsize
Mathematisches Institut, Georg-August-Universit\"at G\"ottingen

Bunsenstra\ss e 3-5, 37073 G\"ottingen, Germany

\texttt{daniele.dona@mathematik.uni-goettingen.de}
\bigskip\bigskip\bigskip
\end{center}

\begin{minipage}{110mm}
\small
\textbf{Abstract.} Following partially a suggestion by Pyber, we prove that the diameter of a product of non-abelian finite simple groups is bounded linearly by the maximum diameter of its factors. For completeness, we include the case of abelian factors and give explicit constants in all bounds.
\medskip

\textbf{Keywords.} Finite simple groups, diameter.
\medskip

\textbf{MSC2010.} 20F69, 20D06.
\end{minipage}
\bigskip

\section{Introduction}

An important area of research in finite group theory in the last decades has been the production of upper bounds for the diameter of Cayley graphs of such groups. Arguably the best known conjecture in the area is Babai's conjecture \cite{BS88}: every non-abelian finite simple group $G$ has diameter $\leq\log^{k}|G|$, where $k$ is an absolute constant; the conjecture is still open, despite great progress towards a solution both for alternating groups and for groups of Lie type.

A more modest question is that of producing bounds for the diameter of direct products of finite simple groups, depending on the diameter of their factors. This is not an idle question, for bounds of this sort have been used more than once as intermediate steps towards the proof of bounds for simple groups themselves: Babai and Seress have done so in \cite[Lemma 5.4]{BS92}, as well as Helfgott more than two decades later in \cite[Lemma 4.13]{He18}. We improve on both results in the following theorem, which also features explicit constants.

\begin{theorem}\label{thmain}
Let $G=\prod_{i=1}^{n}T_{i}$, where the $T_{i}$ are finite simple groups.
\begin{enumerate}[(a)]
\item\label{thmainab} If the $T_{i}$ are all abelian (say $G=\prod_{j=1}^{s}(\mathbb{Z}/p_{j}\mathbb{Z})^{e_{j}}$, where the $p_{j}$ are distinct primes and $e_{j}\geq 1$), then:
\begin{equation*}
\textup{diam}(G)<\frac{2}{3}\max\{e_{j}|1\leq j\leq s\}\prod_{j=1}^{s}p_{j}.
\end{equation*}
\item\label{thmainnonab} If the $T_{i}$ are all non-abelian, call $d=\max\{\textup{diam}(T_{i})|1\leq i\leq n\}$; then:
\begin{equation*}
\textup{diam}(G)<\frac{196}{243}n^{3}\max\{C_{A},C_{L},C_{S}\}(4d+1)+d,
\end{equation*}
where:
\begin{align*}
C_{A}= & \begin{cases} \max\left\{3,\left\lfloor\frac{m}{2}\right\rfloor\right\} & \text{if there are alternating groups among the $T_{i}$} \\ & \text{and where $m$ is their maximum degree,} \\ 0 & \text{if there are no alternating groups among the $T_{i}$,} \end{cases} \\
C_{L}= & \begin{cases} 8(5r+7) & \text{if there are groups of Lie type among the $T_{i}$} \\ & \text{and where $r$ is their maximum untwisted rank,} \\ 0 & \text{if there are no groups of Lie type among the $T_{i}$,} \end{cases} \\
C_{S}= & \begin{cases} 6 & \text{if there are sporadic or Tits groups among the $T_{i}$,} \\ 0 & \text{if there are no sporadic or Tits groups among the $T_{i}$.} \end{cases}
\end{align*}
\item\label{thmainmix} If there are abelian and non-abelian $T_{i}$, write $G=G_{A}\times G_{NA}$, where $G_{A}$ collects the abelian factors and $G_{NA}$ collects the non-abelian ones; then:
\begin{equation*}
\textup{diam}(G)\leq d_{A}+4d_{NA},
\end{equation*}
where $d_{A}=\textup{diam}(G_{A}),d_{NA}=\textup{diam}(G_{NA})$.
\end{enumerate}
\end{theorem}

The result of part \eqref{thmainab} is known and elementary: see \cite[Lemma 5.2]{BS92}, where the constant is marginally worse only due to the fact that sets of generators are not required to be symmetric (cfr. also \cite[Lemma 4.14]{He18}, which treats the case of $G=(\mathbb{Z}/p\mathbb{Z})^{e}$ under this assumption). Part \eqref{thmainmix} is quite natural, given the different (in some sense, opposite) behaviour of abelian and non-abelian factors, as it can be readily observed in its short proof.

Part \eqref{thmainnonab} is where the novelty of the result resides. Dependence on the maximum of the diameter of the components, instead of dependence on their product as Schreier's lemma (see Lemma~\ref{leschreier}) would naturally give us, was already established in \cite[Lemma 5.4]{BS92}: in that case, the diameter was bounded as $O(d^{2})$, where the dependence of the constant on $n$ was polynomial as in our statement. This result was improved in \cite[Lemma 4.13]{He18} to $O(d)$, but only in the case of alternating groups: this was done in part to fix a mistake in the use of the previously available result in Babai-Seress, which is why only alternating groups were considered, as permutation subgroups were the sole concern in both papers; a suggestion by Pyber, reported in Helfgott's paper, points at the results by Liebeck and Shalev \cite{LS01} as a way to prove a bound of $O(d)$ for a product of arbitrary non-abelian finite simple groups.

Indeed, the general approach that we follow in our proof owes its validity to \cite[Thm. 1.6]{LS01}, although we do not explicitly use the statement of that theorem: rather, we closely follow the proof of \cite[Lemma 4.13]{He18} and show that the same reasoning applies to groups of Lie type as well. The way that the lemma is related to Liebeck-Shalev is through the use of the fact that every element in $\text{Alt}(m)$ is a commutator (\cite[Lemma 4.12]{He18}, first proved in \cite[Thm. I]{Mi99}), which is essentially \cite[Thm. 1.6]{LS01} with $w=xyx^{-1}y^{-1}$ and a $c$ that is just equal to $1$ for $\text{Alt}(m)$; the same can be said for all non-abelian finite simple groups (i.e., $c=1$ in general) since Ore's conjecture \cite{Or51} was established to be true in \cite{LOST10}, a fact yet unproved at the time of \cite{LS01}.

\section{Preliminaries}

Before we turn to the proof of Theorem~\ref{thmain}, we will need a certain number of group-theoretic results.

\begin{lemma}[Schreier's lemma]\label{leschreier}
Let $G$ be a finite group, let $N\unlhd G$, and let $S$ be a set of generators of $G$ with $e\in S=S^{-1}$. Then $S^{2d+1}\cap N$ generates $N$, where $d=\textup{diam}(G/N)$.
\end{lemma}

\begin{proof}
This is a standard result dating back to Schreier \cite{Sc27}, written in various fashions across the literature according to the needs of the user; let us prove here the present version.

Calling $\pi:G\rightarrow G/N$ the natural projection, by definition we have $\pi(S)^{d}=G/N$; this equality means that $S^{d}$ contains at least one representative for each coset $gN$ in $G$. For any coset $gN$, choose a representative $\tau(g)\in S^{d}$. Then, for any $h\in N$ and any way to write $h$ as a product of elements $s_{i}\in S$, we have:
\begin{align*}
h= & \ s_{1}s_{2}\ldots s_{k}= \\
= & \ (s_{1}\tau(s_{1})^{-1})\cdot(\tau(s_{1})s_{2}\tau(\tau(s_{1})s_{2})^{-1})\cdot\ldots\cdot(\tau(\tau(\tau(\ldots)s_{k-2})s_{k-1})s_{k}).
\end{align*}
Each element of the form $\tau(x)s_{i}\tau(\tau(x)s_{i})^{-1}$ is contained in $S^{2d+1}\cap N$, so the same can be said about the last element of the form $\tau(x)s_{k}$ (since $h$ itself is in $N$); therefore $S^{2d+1}\cap N$ is a generating set of $N$.
\end{proof}

\begin{proposition}[Ore's conjecture]\label{prore}
Let $G$ be a finite non-abelian simple group. Then, for any $g\in G$, there exist $g_{1},g_{2}\in G$ such that $g=[g_{1},g_{2}]$.
\end{proposition}

\begin{proof}
See \cite{LOST10}, for references to previously known results and for the proof of the final case.
\end{proof}

Notice that, for any finite non-abelian simple group $G$, any nontrivial conjugacy class $C$ must generate the whole $G$ (because $\langle C\rangle$ would be a normal subgroup). This observation justifies the following definition.

\begin{definition}\label{decd}
Let $G$ be a finite non-abelian simple group. The conjugacy diameter $\textup{cd}(G)$ is the smallest $m$ such that $(C\cup C^{-1}\cup\{e\})^{m}=G$ for all nontrivial conjugacy classes $C$.
\end{definition}

We will need to have bounds for $\text{cd}(G)$.

\begin{proposition}\label{prcd}
Let $G$ be a finite non-abelian simple group.
\begin{enumerate}[(a)]
\item\label{prcdAlt} If $G$ is an alternating group of degree $m$, then $\textup{cd}(G)\leq\max\left\{3,\left\lfloor\frac{m}{2}\right\rfloor\right\}$.
\item\label{prcdLie} If $G$ is a group of Lie type of untwisted rank $r$, then $\textup{cd}(G)\leq 8(5r+7)$.
\item\label{prcdSpor} If $G$ is a sporadic group or the Tits group, then $\textup{cd}(G)\leq 6$.
\end{enumerate}
\end{proposition}

\begin{proof}
First of all, $\text{cd}(G)$ is trivially bounded by definition by the \textit{covering number} of $G$, which is defined as $\text{cn}(G)=\min\{m|\forall C\neq\{e\}(C^{m}=G)\}$; therefore it suffices to give bounds for $\text{cn}(G)$.

For \eqref{prcdAlt}, see \cite[Thm. 9.1]{Dv85} (our specific result is credited therein to a manuscript by J. Stavi). For \eqref{prcdLie}, see \cite[Thm. 1]{LL98}. To prove \eqref{prcdSpor}, the sporadic groups all satisfy $\text{cn}(G)\leq 6$: this inequality can be checked directly from \cite[Table 1]{Zi89}; if $G={}^{2}F_{4}(2)'$ is the Tits group, we can show the same inequality using \cite[Lemma 3]{Zi89} and the character values reported in the ATLAS of Finite Groups \cite{CCNPW85}.
\end{proof}

Let us also perform a side computation separately from the proof of the main theorem, so as not to bog down the exposition there.

\begin{lemma}\label{lecomput}
Let $n\geq 1$. Then:
\begin{equation*}
\sum_{i=1}^{n-1}4^{\lceil\log_{2}i\rceil}<\frac{196}{243}n^{3}.
\end{equation*}
\end{lemma}

\begin{proof}
Call $m=\lceil\log_{2}(n-1)\rceil$, and write $n-1=2^{m-1}+l$, where $1\leq l\leq 2^{m-1}$; $\lceil\log_{2}i\rceil=j$ for all $i\in(2^{j-1},2^{j}]$, hence we can rewrite the sum in the statement as:
\begin{align*}
\sum_{i=1}^{n-1}4^{\lceil\log_{2}i\rceil}= & \ 1+\sum_{j=1}^{m-1}4^{j}2^{j-1}+4^{m}l=\frac{1}{2}+\frac{1}{2}\frac{8^{m}-1}{7}+4^{m}\left(2^{\log_{2}(n-1)}-2^{m-1}\right)= \\
= & \ \frac{3}{7}+4^{m}2^{\log_{2}(n-1)}-\frac{3}{7}8^{m}=\frac{3}{7}+2^{2m'}\left(1-\frac{3}{7}2^{m'}\right)(n-1)^{3},
\end{align*}
where $m'=m-\log_{2}(n-1)\in[0,1)$. We have $x^{2}\left(1-\frac{3}{7}x\right)\leq\frac{196}{243}$ for $x\in[1,2)$, and $\frac{3}{7}<\frac{196}{243}(3n^{2}-3n+1)$ for all $n\geq 1$, so the result is proved.
\end{proof}

\section{Proof of the main theorem}

\begin{proof}[Proof of Thm.~\ref{thmain}\ref{thmainab}]
Let $G=(\mathbb{Z}/p_{1}\mathbb{Z})^{e_{1}}\times(\mathbb{Z}/p_{2}\mathbb{Z})^{e_{2}}\times\ldots\times(\mathbb{Z}/p_{s}\mathbb{Z})^{e_{s}}$, with primes $p_{1}<p_{2}<\ldots<p_{s}$; we have:
\begin{equation}\label{eqabprod}
G=A_{1}A_{2}\ldots A_{s}
\end{equation}
(we are using multiplicative notation even if $G$ is abelian) where the $A_{i}$ are any sets such that:
\begin{align}\label{eqabfact}
A_{i,i}= & (\mathbb{Z}/p_{i}\mathbb{Z})^{e_{i}} & A_{i,j}= & (0)^{e_{j}} \ \ \ (\forall j<i)
\end{align}
where $A_{i,j}$ is the projection of $A_{i}$ to the $j$-th component of $G$.

Let $S$ be a set of generators of $G$ with $e\in S=S^{-1}$: $\{t^{p_{1}\ldots p_{i-1}}|t\in S\}\subseteq S^{p_{1}\ldots p_{i-1}}$ has elements that are all $0$ on the first $i-1$ components of $G$ and that still generate the $i$-th one since $(p_{1}\ldots p_{i-1},p_{i})=1$; from now on, let us focus exclusively on the $i$-th component. $(\mathbb{Z}/p_{i}\mathbb{Z})^{e_{i}}$ is also a vector space over $\mathbb{Z}/p_{i}\mathbb{Z}$, so there must be $e_{i}$ generators that also form a basis: any element of the space can be written as a linear combination of those generators with coefficients in $\left[-\left\lfloor\frac{p_{i}}{2}\right\rfloor,\left\lfloor\frac{p_{i}}{2}\right\rfloor\right]$, which corresponds to a word of length $\leq e_{i}\left\lfloor\frac{p_{i}}{2}\right\rfloor$; thus, each set $A_{i}$ with the properties in \eqref{eqabfact} is covered in $e_{i}\left\lfloor\frac{p_{i}}{2}\right\rfloor p_{1}\ldots p_{i-1}$ steps. This fact and \eqref{eqabprod} imply that $G$ has diameter bounded by:
\begin{equation}\label{eqabprimes}
\sum_{i=1}^{s}\left(e_{i}\left\lfloor\frac{p_{i}}{2}\right\rfloor\prod_{j=1}^{i-1}p_{j}\right)\leq\frac{1}{2}\max\{e_{j}|1\leq j\leq s\}\prod_{j=1}^{s}p_{j}\cdot\sum_{i=1}^{s}\left(\prod_{j=i+1}^{s}\frac{1}{p_{j}}\right).
\end{equation}
The sum in \eqref{eqabprimes} is maximized when each $p_{j}$ is the $j$-th prime number: for $s=1$ the sum is $1$ and for $s=2$ it is bounded by $\frac{4}{3}$; for $s\geq 3$, we use $p_{s}\geq 5$ and $p_{j}\geq 3$ for all $1<j<s$, so that the sum is bounded by $1+\frac{1}{5}\frac{1}{1-\frac{1}{3}}=\frac{13}{10}$. The result follows.
\end{proof}

\begin{proof}[Proof of Thm.~\ref{thmain}\ref{thmainnonab}]
Calling $G_{j}=\prod_{i=1}^{j}T_{i}$, we have natural projections $\pi_{j}:G=G_{n}\rightarrow G_{j}$ and $\rho_{j_{1},j_{2}}:G_{j_{1}}\rightarrow T_{j_{2}}$ for any $j_{1}\geq j_{2}$. As in \eqref{eqabprod}, we write $G$ as a product of subsets $A_{i}$ with $\rho_{n,i}(A_{i})=T_{i}$ and $\rho_{n,j}(A_{i})=\{e\}$ for all $j<i$, and our aim is to cover each one of them.

Suppose that we have two subsets $X_{1},X_{2}$ of $G$ for which $\rho_{n,i}(X_{1})=\rho_{n,i}(X_{2})=T_{i}$ for some fixed $i\in\{1,\ldots,n\}$ and that have $\rho_{n,j_{1}}(X_{1})=\{e\}=\rho_{n,j_{2}}(X_{2})$ for all $j_{1}\in I_{1},j_{2}\in I_{2}$, where $I_{1},I_{2}$ are two subsets of indices in $\{1,\ldots,n\}\setminus\{i\}$: then, the set $X=\{[x_{1},x_{2}]|x_{1}\in X_{1},x_{2}\in X_{2}\}$ has $\rho_{n,i}(X)=T_{i}$ by Proposition~\ref{prore} (Ore's conjecture) and $\rho_{n,j}(X)=\{e\}$ for all $j\in I_{1}\cup I_{2}$. Now consider the set of indices $I=\{1,\ldots,i-1\}$: if $|I|>1$ we can partition $I$ into two parts of size $\left\lfloor\frac{|I|}{2}\right\rfloor,\left\lceil\frac{|I|}{2}\right\rceil$, then partition each part $I'$ with $|I'|>1$ into two new parts again of size $\left\lfloor\frac{|I'|}{2}\right\rfloor,\left\lceil\frac{|I'|}{2}\right\rceil$, and continue until we reach a subdivision where all sets have size $1$; the tree of partitions that we constructed to reach this subdivision will have exactly $\lceil\log_{2}|I|\rceil$ layers. Notice that, given any two parts $I_{1},I_{2}$ inside the tree, if we have two subsets $X_{1},X_{2}$ (as described before) that are covered by a certain $S^{a}$, the resulting set $X$ will be covered by $S^{4a}$: this observation, together with the information about the layers, tells us that if we can cover sets $X_{i,j}$ with $\rho_{n,i}(X_{i,j})=T_{i}$ and $\rho_{n,j}(X_{i,j})=\{e\}$ in $a$ steps (for a fixed $i>1$ and all $j<i$) then we are able to cover a set $A_{i}$ defined as at the beginning of the proof in $4^{\lceil\log_{2}(i-1)\rceil}a$ steps as well.

Let us start now with a generating set $S$ with $e\in S=S^{-1}$ and fix two indices $i\geq j$: $\pi_{i}(S)$ is a set of generators for $G_{i}$, and the set $\pi_{i}(S)^{2d+1}$ contains generators for the whole $T_{1}\times\ldots\times T_{j-1}\times\{e\}\times T_{j+1}\times\ldots\times T_{i}=G_{i}\cap\ker(\rho_{i,j})$ by Lemma~\ref{leschreier} (Schreier's lemma), where $d$ is as in the statement. In particular, there is an element $x\in S^{2d+1}$ with $\rho_{n,i}(x)\neq e$ and $\rho_{n,j}(x)=e$; by hypothesis $\rho_{n,i}(S^{d})=T_{i}$, which means that there is a set $S'=\{yxy^{-1}|y\in S^{d}\}\cup\{yx^{-1}y^{-1}|y\in S^{d}\}\cup\{e\}\subseteq S^{4d+1}$ with $\rho_{n,i}(S')=C\cup C^{-1}\cup\{e\}$ and $\rho_{n,j}(S')=\{e\}$, where $C$ is the conjugacy class of $\rho_{n,i}(x)$. By Proposition~\ref{prcd}, $\rho_{n,i}(S'^{\max\left\{3,\left\lfloor\frac{m_{i}}{2}\right\rfloor\right\}})=T_{i}$ if $T_{i}=\text{Alt}(m_{i})$, $\rho_{n,i}(S'^{8(5r_{i}+7)})=T_{i}$ if $T_{i}$ is of Lie type of untwisted rank $r_{i}$, and $\rho_{n,i}(S'^{6})=T_{i}$ otherwise; in all three cases, the projection to $T_{j}$ is still $\{e\}$, therefore we managed to cover a set $X_{i,j}$ of the aforementioned form.

A set $A_{1}$ is reached in $d$ steps, hence the final count for the whole $G$ following the reasoning above is:
\begin{equation*}
\text{diam}(G)\leq d+\sum_{i=2}^{n}4^{\lceil\log_{2}(i-1)\rceil}x_{i}(4d+1),
\end{equation*}
where $x_{i}$ is either $\max\left\{3,\left\lfloor\frac{m_{i}}{2}\right\rfloor\right\}$, $8(5r_{i}+7)$ or $6$, accordingly. The result follows by Lemma~\ref{lecomput}.
\end{proof}

A note on the connection between the proof given above and \cite{LS01}. As mentioned before, Pyber pointed at \cite{LS01} as a way to prove linear dependence on $d$ for products of arbitrary non-abelian finite simple groups. In particular, \cite[Thm. 1.6]{LS01} seems to fit the bill: it states that for any word $w$ that is not a law in a finite simple group $T$ there is $c_{w}\in\mathbb{N}$, depending on $w$ but not on $T$, such that any element of $T$ can be written as a product of at most $c_{w}$ values of $w$. We use this property, in disguise, when we want to pass from two subsets being indentically $e$ at indices $I_{1},I_{2}$ and filling an entire component $T_{i}$ to a third subset that also fills the same component and is $e$ for the whole $I_{1}\cup I_{2}$: the creation of the new subset is made possible by taking $c_{w}$ values of a word $w$, so that $T_{i}$ remains filled, where $w$ has two distinct letters $x_{1},x_{2}$ and presents the same number of $x_{i}$ and $x_{i}^{-1}$ for $i\in\{1,2\}$, so that when any one $x_{i}$ is equal to $e$ on a given factor of the product $G$ the result is $e$ on that factor; in our case, $w$ was the shortest nontrivial word with these characteristics, namely the commutator $[x_{1},x_{2}]=x_{1}x_{2}x_{1}^{-1}x_{2}^{-1}$ (not a law for any non-abelian group), and $c_{w}=1$ by Ore's conjecture. In this sense $w=[x_{1},x_{2}]$ is also computationally the best word we can expect, for it yields the lowest possible value of $|w|c_{w}$, the $4$ that we find in Lemma~\ref{lecomput}.

\begin{proof}[Proof of Thm.~\ref{thmain}\ref{thmainmix}]
Define the two projections $\pi_{A},\pi_{NA}$ in the obvious way; for any generating set $S$ of $G$, by definition there are a subset $X_{A}\subseteq S^{d_{A}}$ with $\pi_{A}(X_{A})=G_{A}$ and a subset $X_{NA}\subseteq S^{d_{NA}}$ with $\pi_{NA}(X_{NA})=G_{NA}$, and then:
\begin{equation*}
G=X_{A}[X_{NA},X_{NA}]\subseteq S^{d_{A}+4d_{NA}},
\end{equation*}
again by the fact that $[T,T]=T$ for non-abelian finite simple groups by Ore's conjecture and $[T,T]=\{e\}$ for abelian groups.
\end{proof}

\section{Concluding remarks}

One could wonder how tight the inequalities in Theorem~\ref{thmain} are. The results are essentially in line with what is generally expected from the behaviour of the diameter of finite groups. The abelian case is tight up to constant: for the group $G(x)=\prod_{p\leq x}\mathbb{Z}/p\mathbb{Z}$ (nontrivial for $x\geq 2$) one generator $s=(1,1,\ldots,1)$ is enough, and then the diameter of $\text{Cay}(G(x),\{s,s^{-1},e\})$ is $\frac{1}{2}|G(x)|$; the fact that abelian groups behave in the worst possible way, i.e. linearly in the size of the group, should not be a surprise for anyone.

The non-abelian bound of case \eqref{thmainnonab} also matches what is anticipated in general. Babai's conjecture posits a polylogarithmic bound on the diameter of finite simple groups: the natural extension to direct products of such groups would suggest a bound of the form $n^{k}d$, which is exactly what we have obtained. Case \eqref{thmainmix} also fits into the same idea, as a product $|G|=|G_{A}||G_{NA}|$ becomes a sum of the corresponding diameters.

The dependence on $d$ in Theorem~\ref{thmain}\ref{thmainnonab} is almost best possible by definition (we cannot drop the ``almost'', as $m,r$ are not independent from $d$). It would be more interesting to understand which power of $n$ is the correct one: here we have proved $O_{m,r,d}(n^{3})$, and we can quickly show that the bound is $\Omega_{m,r,d}(n)$, as illustrated in the following example.

\begin{example}\label{exlow}
If $G=(\text{Alt}(m))^{n}$ then $\text{diam}(G)=\Omega(m^{2}n)$. We prove it for $m\geq 5$ odd and $n$ even, but the proof is analogous for the general case.

Consider the two permutations $\sigma=(1\ 2\ 3\ \ldots\ m)$ and $\tau=(1\ 2\ 3\ \ldots\ m-2)$; they generate $\text{Alt}(m)$, and the elements:
\begin{align*}
s_{0}= & \ (\sigma,\sigma,\ldots,\sigma,\sigma), \\
s_{1}= & \ (\tau,\sigma,\ldots,\sigma,\sigma), \\
s_{2}= & \ (\sigma,\tau,\ldots,\sigma,\sigma), \\
 & \ \ldots \\
s_{n}= & \ (\sigma,\sigma,\ldots,\sigma,\tau)
\end{align*}
generate $G$. Let $S=\{e\}\cup\{s_{i},s_{i}^{-1}\}_{0\leq i\leq n}$: to prove the lower bound on the diameter of $G$, we construct a function $f:G\rightarrow\mathbb{N}$ such that there are two elements $g_{1},g_{2}\in G$ with $|f(g_{1})-f(g_{2})|$ large and such that $|f(g)-f(gs)|$ is small for any $g\in G,s\in S$; this is a known technique to prove lower bounds for the diameter of $\text{Sym}(m)$, as shown for instance in \cite[Prop. 3.6]{Ta11}.

Call $c(g,i,j)=(g(i))(j)$ the image of $j\in\{1,\ldots,m\}$ under the $i$-th component of $g\in G$, for $1\leq i\leq n$; define:
\begin{equation*}
f(g)=\sum_{j=1}^{m}\sum_{i=1}^{n}||c(g,i+1,j)-c(g,i,j)||_{\mathbb{Z}/m\mathbb{Z}},
\end{equation*}
where $||a||_{\mathbb{Z}/m\mathbb{Z}}=\min\{a,m-a\}$ (in the case $i=n$, $c(g,n+1,j)$ means $c(g,1,j)$). First, $f(e)=0$; also, if we call $e_{m}$ the identity element in $\text{Alt}(m)$ and $\eta=\left(1\ \frac{m+1}{2}\right)\left(2\ \frac{m+3}{2}\right)\ldots\left(\frac{m-1}{2}\ m-1\right)$, for $g\in G$ that has $e_{m}$ at all odd components and $\eta$ at all even ones we have $f(g)=\frac{1}{2}(m-1)^{2}n$. Finally, notice that $\sigma$ simply adds $1$ modulo $m$ to all the elements of $\{1,\ldots,m\}$, so that $f(g)=f(gs_{0}^{\pm 1})$, while $\tau$ is defined so that it adds $1$ for $m-3$ elements, adds $3$ (modulo $m$) for one element and fixes two elements, which means that $|f(g)-f(gs_{i}^{\pm 1})|\leq 10$; these facts taken together imply that $\text{diam}(G,S)\geq\frac{1}{20}(m-1)^{2}n$.
\end{example}

The correct (or even expected) order of magnitude for a bound of the form $\text{diam}(G)=O_{m,r}(n^{k}d)$ for a generic product $G$ is not known to the author, besides knowing that $1\leq k\leq 3$ by Theorem~\ref{thmain} and Example~\ref{exlow}.

\end{document}